\documentclass[smallcondensed]{amsart}                     
\usepackage{graphicx}
\usepackage{mathptmx}      
\usepackage{amsmath,amssymb,amsbsy,
}
\usepackage{
eucal,enumerate}
\usepackage{latexsym}
\newcommand{\X}{\mathfrak{X}}
\newcommand{\F}{\mathcal{F}}

\newcommand{\U}{\mathcal{U}}
\newcommand{\T}{\mathcal{T}}
\newcommand{\dt}{\frac{\mathrm{d}}{\mathrm{d}t}}
\newcommand{\ddt}{\partial_t}
\newcommand{\n}{\nabla}
\newcommand{\f}{\varphi}
\newcommand{\D}{{\rm d}}
\newcommand{\al}{\alpha}
\newcommand{\bt}{\beta}
\newcommand{\gm}{\gamma}
\newcommand{\lm}{\lambda}
\newcommand{\ta}{\theta}
\newcommand{\om}{\omega}
\DeclareMathOperator{\Id}{Id}
\newcommand{\grad}{\mathrm{grad}}
\newcommand{\sx}{\mathop{\mathfrak{S}}\limits_{x,y,z}}
\newcommand{\M}{(M,\allowbreak\f,\allowbreak\xi,\allowbreak\eta,g)}

\newcommand{\s}{\mathfrak{S}}
\newcommand{\R}{\mathbb{R}}
\newcommand{\C}{\mathbb{C}}

\newcommand{\propref}[1]{Pro\-po\-si\-ti\-on~\ref{#1}}
\newcommand{\lemref}[1]{Lem\-ma~\ref{#1}}

\newcommand{\corref}[1]{Corollary~\ref{#1}}
\newtheorem{thm}{Theorem}
\newtheorem{lem}[thm]{Lemma}
\newtheorem{prop}[thm]{Proposition}
\newtheorem{cor}[thm]{Corollary}
\newtheorem{defn}{Definition}
\begin{document}

\title[Canonical-type connection]{Canonical-type connection on almost contact manifolds with
B-metric
}


\author{Mancho Manev \and Miroslava Ivanova
}


\address{
                University of Plovdiv,
                Faculty of Mathematics and Informatics,
                Department of Algebra and Geometry\\
                236 Bulgaria Blvd,
                Plovdiv 4003,
                Bulgaria,
              Tel.: +359-32-261744,
              Fax: +359-32-261794}
              \email{mmanev@uni-plovdiv.bg, mirkaiv@uni-plovdiv.bg}        


\begin{abstract}
The canonical-type connection on the almost contact manifolds with
B-met\-ric is constructed. It is proved that its torsion is
invariant with respect to a subgroup of the general conformal
transformations of the almost contact B-metric structure. The
basic classes of the considered manifolds are characterized in
terms of the torsion of the canonical-type connection.
\end{abstract}


\maketitle

\section*{Introduction}

In the differential geometry of the manifolds with additional
structures, there are important the so-called natural connections,
i.e. linear connections with torsion such that the additional
structures are parallel with respect to them. There exists a
significant interest to these natural connections which have some
additional geometric or algebraic properties, for instance about
their torsion.

On an almost Hermitian manifold $(M,J,g)$ there exists a unique
natural connection $\n^C$ with a torsion $T$ such that
$T(J\cdot,J\cdot)=-T(\cdot,\cdot)$. This connection is known as
the canonical Hermitian connection or the Chern connection. An
example of the natural Hermitian connection is the first canonical
connection of Lichnerowicz $\n^L$ \cite{Li1,Li2}.
According to \cite{Gaud}, there exists a one-parameter family of
canonical Hermitian connections $\n^t=t\n^C+(1-t)\n^L$. The
connection $\n^t$ obtained for $t=-1$ is called the Bismut
connection or the KT-connection, which is characterized with a
totally skew-symmetric torsion. The latter connection with a
closed torsion 3-form has applications in type II string theory
and in 2-dimensional supersymmetric $\sigma$-models
\cite{GHR,Stro,IvPapa}. In \cite{Fri-Iv2} and \cite{Fri-Iv} all
almost contact metric, almost Hermitian and $G_2$-structures
admitting a connection with totally skew-symmetric torsion tensor
are described.

Natural connections of canonical type are considered on the
Riemannian almost product manifolds in
\cite{Dobr11-1,Dobr11-2,MekDobr} and on the almost complex
manifolds with Norden metric in \cite{GaMi87,GaGrMi,Mek-JTech}.
The connection in \cite{GaGrMi} is the so-called B-connection,
which is studied on the class of the locally conformal K\"ahlerian
manifolds with Norden metric.

In the present paper\footnote{Partially supported by project
NI11-FMI-004 of the Scientific Research Fund, Paisii Hilendarski
University of Plovdiv, Bulgaria and the German Academic Exchange
Service (DAAD)} we consider natural connections (i.e. preserving
the structure) of canonical type on the almost contact manifolds
with B-metric. These manifolds are the odd-dimensional extension
of the almost complex manifolds with Nor\-den metric and the case
with indefinite metrics corresponding to the almost contact metric
manifolds.

The paper is organized as follows. In Sec.~1 we give some
necessary facts about the considered manifolds.
In Sec.~2 we define a natural connection of canonical type on an
almost contact manifold with B-metric. We determine the class of
the considered manifolds where this connection and a known natural
connection coincide.
In Sec.~3 we consider the group $G$ of the general conformal
transformations of the almost contact B-metric structure. We
determine the invariant class of the considered manifolds and a
tensor invariant of the group $G$.
In Sec.~4 we establish that the torsion of the canonical-type
connection is invariant exactly in the subgroup $G_0$ of $G$. We
characterize the basic classes of the considered manifolds by the
torsion of the canonical-type connection.
In Sec.~5 we determine the class of the almost contact B-metric
manifolds of Sasakian type and supply a relevant example.

\section{Almost Contact Manifolds with B-metric}\label{sec:1}

Let $(M,\f,\xi,\eta)$ be an almost contact manifold,  i.e.  $M$ is
a $(2n+1)$-dimensional diffe\-ren\-tia\-ble manifold with an
almost contact structure $(\f,\xi,\eta)$ consisting of an
endomorphism $\f$ of the tangent bundle, a vector field $\xi$ and
its dual 1-form $\eta$ such that the following algebraic relations
are satisfied:
\begin{equation}\label{str1}
\f\xi = 0,\quad \f^2 = -\Id + \eta \otimes \xi,\quad
\eta\circ\f=0,\quad \eta(\xi)=1.
\end{equation}
Further, let us endow the almost contact manifold
$(M,\f,\xi,\eta)$ with a pseudo-Riemannian metric $g$ of signature
$(n,n+1)$ determined by
\begin{equation}\label{str2}
g(\f x, \f y ) = - g(x, y ) + \eta(x)\eta(y)
\end{equation}
for arbitrary $x$, $y$ of the algebra $\X(M)$ on the smooth vector
fields on $M$.
Then $(M,\f,\xi,\eta,g)$ is called an almost contact manifold with
B-metric or an \emph{almost contact B-metric manifold}.

Further, $x$, $y$, $z$ will stand for arbitrary elements of
$\X(M)$.

The associated metric $\tilde{g}$ of $g$ on $M$ is defined by
$\tilde{g}(x,y)=g(x,\f y)\allowbreak+\eta(x)\eta(y)$. Both metrics
$g$ and $\tilde{g}$ are necessarily of signature $(n,n+1)$. The
manifold $(M,\f,\xi,\eta,\tilde{g})$ is also an almost contact
B-metric manifold.

Let us remark that the $2n$-dimensional contact distribution
$H=\ker(\eta)$ generated by the contact 1-form $\eta$ can be
considered as the horizontal distribution of the sub-Riemannian
manifold $M$. Then $H$ is endowed with an almost complex structure
determined as $\f|_H$ -- the restriction of $\f$ on $H$, as well
as a Norden metric $g|_H$, i.e.
$g|_H(\f|_H\cdot,\f|_H\cdot)=-g|_H(\cdot,\cdot)$. Moreover, $H$
can be considered as a $n$-dimensional complex Riemannian manifold
with a complex Riemannian metric $g^{\C}=g|_H+i\tilde{g}|_H$
\cite{GaIv}.

The structural group of the almost contact B-metric manifolds is
$\allowbreak\bigl(GL(n,\mathbb{C})\cap
O(n,n)\bigr)\allowbreak\times I_1$, i.e. it consists of real
square matrices of order $2n+1$ of the following type
\[
\left(%
\begin{array}{r|c|c}
  A & B & \vartheta^T\\ \hline
  -B & A & \vartheta^T\\ \hline
  \vartheta & \vartheta & 1 \\
\end{array}%
\right),\qquad %
\begin{array}{l}
  A^TA-B^TB=I_n,\\%
  B^TA+A^TB=O_n,
\end{array}%
\quad A, B\in GL(n;\mathbb{R}),
\]
where $\vartheta$ and its transpose $\vartheta^T$ are the zero row
$n$-vector and the zero column $n$-vector; $I_n$ and $O_n$ are the
unit matrix and the zero matrix of size $n$, respectively.

\subsection{The structural tensor $F$}

The covariant derivatives of $\f$, $\xi$, $\eta$ with respect to
the Levi-Civita connection $\n$ play a fundamental role in the
differential geometry on the almost contact manifolds.  The
structural tensor $F$ of type (0,3) on $\M$ is defined by
\begin{equation}\label{F=nfi}
F(x,y,z)=g\bigl( \left( \nabla_x \f \right)y,z\bigr).
\end{equation}
It has the following properties:
\begin{equation}\label{F-prop}
\begin{split}
F(x,y,z)&=F(x,z,y)=F(x,\f y,\f z)+\eta(y)F(x,\xi,z)
+\eta(z)F(x,y,\xi).
\end{split}
\end{equation}
The relations of $\n\xi$ and $\n\eta$ with $F$ are:
\begin{equation*}\label{Fxieta}
    \left(\n_x\eta\right)y=g\left(\n_x\xi,y\right)=F(x,\f y,\xi).
\end{equation*}

The following 1-forms are associated with $F$:
\begin{equation}\label{titi}
\begin{array}{c}
\ta(z)=g^{ij}F(e_i,e_j,z),\quad \ta^*(z)=g^{ij}F(e_i,\f
e_j,z),\quad \om(z)=F(\xi,\xi,z),
\end{array}
\end{equation}
where $g^{ij}$ are the components of the inverse matrix of $g$
with respect to a basis $\left\{e_i;\xi\right\}$
$(i=1,2,\dots,2n)$ of the tangent space $T_pM$ of $M$ at an
arbitrary point $p\in M$. Obviously, the equality $\om(\xi)=0$ and
the following relation are always valid:
\begin{equation}\label{tata*}
\ta^*\circ\f=-\ta\circ\f^2.
\end{equation}

A classification of the almost contact B-metric manifolds with
respect to $F$ is given in \cite{GaMiGr}. This classification
includes eleven basic classes $\F_1$, $\F_2$, $\dots$, $\F_{11}$.
Their intersection is the special class $\F_0$ determined by the
condition $F(x,y,z)=0$. Hence $\F_0$ is the class of almost
contact B-metric manifolds with $\n$-parallel structures, i.e.
$\n\f=\n\xi=\n\eta=\n g=\n \tilde{g}=0$.

Further we use the following characteristic conditions of the
basic classes:
\begin{equation}\label{Fi}
\begin{array}{rl}
\F_{1}:\quad &F(x,y,z)=\frac{1}{2n}\bigl\{g(x,\f y)\ta(\f z)+g(\f
x,\f y)\ta(\f^2 z)
\bigr\}_{(y\leftrightarrow z)};\\
\F_{2}:\quad &F(\xi,y,z)=F(x,\xi,z)=0,\quad
              \sx F(x,y,\f z)=0,\quad \ta=0;\\
\F_{3}:\quad &F(\xi,y,z)=F(x,\xi,z)=0,\quad
              \sx F(x,y,z)=0;\\
\F_{4}:\quad &F(x,y,z)=-\frac{\ta(\xi)}{2n}\bigl\{g(\f x,\f y)\eta(z)+g(\f x,\f z)\eta(y)\bigr\};\\
\F_{5}:\quad &F(x,y,z)=-\frac{\ta^*(\xi)}{2n}\bigl\{g( x,\f y)\eta(z)+g(x,\f z)\eta(y)\bigr\};\\
\F_{6/7}:\quad &F(x,y,z)=F(x,y,\xi)\eta(z)+F(x,z,\xi)\eta(y),\quad \\
                &F(x,y,\xi)=\pm F(y,x,\xi)=-F(\f x,\f y,\xi),\quad \ta=\ta^*=0; \\
\F_{8/9}:\quad &F(x,y,z)=F(x,y,\xi)\eta(z)+F(x,z,\xi)\eta(y),\\
                &F(x,y,\xi)=\pm F(y,x,\xi)=F(\f x,\f y,\xi); \\
\F_{10}:\quad &F(x,y,z)=F(\xi,\f y,\f z)\eta(x); \\
\F_{11}:\quad
&F(x,y,z)=\eta(x)\left\{\eta(y)\om(z)+\eta(z)\om(y)\right\},
\end{array}
\end{equation}
where (for the sake of brevity) we use the following denotations:
$\{A(x,y,z)\}_{(x\leftrightarrow y)}$ --- instead of
$\{A(x,y,z)+A(y,x,z)\}$ for any tensor $A(x,y,z)$; $\s$ --- for
the cyclic sum by three arguments; and the former and latter
subscripts of $\F_{i/j}$ correspond to upper and down signs plus
or minus, respectively.

\subsection{The Nijenhuis tensor $N$}

An almost contact structure $(\f,\xi,\eta)$ on $M$ is called
\emph{normal} and respectively $(M,\f,\xi,\eta)$ is a \emph{normal
almost contact manifold} if the corresponding almost complex
structure $J$ generated on $M'=M\times \R$ is integrable (i.e.
$M'$ is a complex manifold) \cite{SaHa}. The almost contact
structure is normal if and only if the Nijenhuis tensor of
$(\f,\xi,\eta)$ is zero \cite{Blair}.

The Nijenhuis tensor $N$ of the almost contact structure is
defined by
\[
N := [\f, \f]+ \D{\eta}\otimes\xi,
\]
where $[\f, \f](x, y)=\left[\f x,\f
y\right]+\f^2\left[x,y\right]-\f\left[\f x,y\right]-\f\left[x,\f
y\right]$ and $\D{\eta}$ is the exterior derivative of the 1-form
$\eta$. Obviously, $N$ is an antisymmetric tensor, i.e.
$N(x,y)=-N(y,x)$. Hence, using $[x,y]=\n_xy-\n_yx$ and
$\D{\eta}(x,y)=\left(\n_x\eta\right)y-\left(\n_y\eta\right)x$, the
tensor $N$ has the following form in terms of the covariant
derivatives with respect to the Levi-Civita connection $\n$:
\begin{equation}\label{N}
\begin{split}
N(x,y)&=\left(\n_{\f x}\f\right)y-\f\left(\n_{x}\f\right)y+\left(\n_{x}\eta\right)y\cdot\xi\\
& -\left(\n_{\f
y}\f\right)x+\f\left(\n_{y}\f\right)x-\left(\n_{y}\eta\right)x\cdot\xi.
\end{split}
\end{equation}
The corresponding Nijenhuis tensor of type (0,3) on $\M$ is
defined by $N(x,y,z)\allowbreak{}=g\left(N(x,y),z\right)$. Then,
from \eqref{N} and \eqref{F=nfi}, we have
\begin{equation}\label{NF}
\begin{split}
N(x,y,z)&=\bigl\{F(\f x,y,z)-F(x,y,\f z)+F(x,\f
y,\xi)\eta(z)\bigr\}_{[x\leftrightarrow y]}.
\end{split}
\end{equation}
where we use the denotation $\left\{A(x, y,
z)\right\}_{[x\leftrightarrow y]}$ instead of $\{A(x, y, z) - A(y,
x, z)\}$ for any tensor $A(x, y, z)$.

\begin{lem}\label{lem-N}
The Nijenhuis tensor on an arbitrary almost B-metric manifold has
the following properties:
\begin{gather}
N(\f x, \f y,\f z)=-N(\f^2 x, \f^2 y,\f z)=N(\f x, \f^2 y,\f^2 z),\nonumber\\
N(\f^2 x, \f y,\f z)=N(\f x, \f^2 y,\f z)=-N(\f x, \f y,\f^2
z).\nonumber
\end{gather}
\end{lem}
\begin{proof}
Bearing in mind properties \eqref{F-prop} of $F$ and relation
\eqref{NF}, the equalities in the first line of the statement
follow. They imply the equalities in the last line by virtue of
\eqref{str1} and \eqref{str2}.
\end{proof}

\begin{lem}\label{lem-Nfifi}
The class $\U_0=
\F_1\oplus\F_2\oplus\F_4\oplus\F_5\oplus\F_6\oplus\F_8\oplus\F_9\oplus\F_{10}\oplus\F_{11}$
of the almost contact B-metric manifolds is determined by the
condition $N(\f\cdot,\f\cdot)=0$.
\end{lem}
\begin{proof}
The statement follows from the following form of the tensor $N$
for each of the basic classes $\F_i$ $(i=1,2,\dots,11)$ of $M=\M$:
\begin{equation}\label{N-1-11}
\begin{array}{ll}
N(x,y)=0,
\quad &M\in\F_1\oplus\F_2\oplus\F_4\oplus\F_5\oplus\F_6;\\
N(x,y)=2\left\{\left(\n_{\f
x}\f\right)y-\f\left(\n_{x}\f\right)y\right\},
\quad &M\in\F_3;\\
N(x,y)=4\left(\n_{x}\eta\right)y\cdot \xi ,
\quad &M\in\F_7;\\
N(x,y)=2\left\{\eta(x)\n_{y}\xi-\eta(y)\n_{x}\xi\right\},
\quad &M\in\F_8\oplus\F_9;\\
N(x,y)=-\eta(x)\f\left(\n_{\xi}\f\right)y+\eta(y)\f\left(\n_{\xi}\f\right)x,
\quad &M\in\F_{10};\\
N(x,y)=\left\{\eta(x)\om(\f y)-\eta(y)\om(\f x)\right\}, \quad
&M\in\F_{11}.
\end{array}
\end{equation}
The calculations are made using \eqref{NF} and \eqref{Fi}.
\end{proof}

\section{$\f$-Canonical Connection}

\begin{defn}
A linear connection $D$ is called a \emph{natural connection} on
the manifold  $(M,\f,\allowbreak\xi,\eta,g)$ if the almost contact
structure $(\f,\xi,\eta)$ and the B-metric $g$ are parallel with
respect to $D$,  i.e.  $D\f=D\xi=D\eta=Dg=0$.
\end{defn}
As a corollary, the associated metric $\tilde{g}$ is also parallel
with respect to a natural connection $D$ on $\M$.

According to \cite{ManIv36}, a necessary and sufficient condition
a linear connection $D$ to be natural on $\M$ is $D\f=Dg=0$.

If $T$ is the torsion of $D$, i.e. $T(x,y)=D_x y-D_y x-[x, y]$,
then the corresponding tensor of type (0,3) is determined by
$T(x,y,z)=g(T(x,y),z)$.

Let us denote the difference between the natural connection $D$
and the Levi-Civita connection $\n$ of $g$ by $Q(x,y)=D_xy-\n_xy$
and the corresponding tensor of type (0,3) --- by
$Q(x,y,z)=g\left(Q(x,y),z\right)$.

It is easy to establish  (see, e.g. \cite{Man31}) that
%
a linear connection $D$ is a natural connection on an almost contact B-metric manifold if and only if %
\begin{equation}\label{D-can}
 Q(x,y,\f z)-Q(x,\f y,z)=F(x,y,z),\qquad
 Q(x,y,z)=-Q(x,z,y).
\end{equation}
%
Therefore, according $T(x,y)=Q(x,y)-Q(y,x)$, we have the equality
of Hayden's theorem \cite{Hay}
\begin{equation*}\label{Hay}
Q(x,y,z)=\frac{1}{2}\left\{T(x,y,z)-T(y,z,x)+T(z,x,y)\right\}.
\end{equation*}

\begin{defn}\label{defn-canonical}
A natural connection $D$ is called a \emph{$\f$-canonical
connection} on the manifold $(M,\f,\xi,\allowbreak\eta,g)$ if the
torsion tensor $T$ of $D$ satisfies the following identity:
\begin{equation}\label{T-can}
\begin{split}
    &\bigl\{T(x,y,z)-T(x,\f y,\f z)
    -\eta(x)\left\{T(\xi,y,z)
    -T(\xi, \f y,\f z)\right\}\\
    &-\eta(y)\left\{T(x,\xi,z)-T(x,z,\xi)-\eta(x)T(z,\xi,\xi)\right\}\bigr\}_{[y\leftrightarrow z]}=0.
\end{split}
\end{equation}
\end{defn}

Let us remark that the restriction of the $\f$-canonical
connection $D$ of the manifold $\M$ on the contact distribution
$H$ is the unique canonical connection of the corresponding almost
complex manifold with Norden metric, studied in \cite{GaMi87}.

In \cite{ManGri2}, it is introduced a natural connection on $\M$,
defined by
\begin{equation}\label{fB}
    \n^0_xy=\n_xy+Q^0(x,y),
\end{equation}
where    $Q^0(x,y)=\frac{1}{2}\bigl\{\left(\n_x\f\right)\f
y+\left(\n_x\eta\right)y\cdot\xi\bigr\}-\eta(y)\n_x\xi$.
Therefore, we have
\begin{equation}\label{Q0}
Q^0(x,y,z)=\frac{1}{2}\left\{F(x,\f y,z)+\eta(z)F(x,\f
y,\xi)-2\eta(y)F(x,\f z,\xi)\right\}.
\end{equation}
The torsion of the $\f$B-connection has the following form
\begin{equation}\label{T0}
\begin{split}
T^0(x,y,z)=\frac{1}{2}\bigl\{&F(x,\f y,z)+\eta(z)F(x,\f y,\xi)
+2\eta(x)F(y,\f z,\xi)\bigr\}_{[x\leftrightarrow y]}.
\end{split}
\end{equation}

In \cite{ManIv37}, the connection determined by \eqref{fB} is
called a \emph{$\f$B-connection}. It is studied for some classes
of $\M$ in \cite{ManGri2,Man3,Man4,ManIv37}. The restriction of
the $\f$B-connection on $H$ coincides with the B-connection on an
almost complex manifold with Norden metric, studied for the class
of the locally conformal K\"ahlerian manifolds with Norden metric
in \cite{GaGrMi}.

We construct a linear connection $\n'$ as follows:
\begin{equation}\label{D=nQ}
g(\n'_xy,z)=g(\n_xy,z)+Q'(x,y,z),
\end{equation}
where
\begin{equation}\label{Q-can-F}
\begin{split}
Q'(x,y,z)
=Q^0(x,y,z)-\frac{1}{8}\left\{N(\f^2 z,\f^2 y,\f^2 x)+2N(\f z,\f
y,\xi)\eta(x)\right\}.
\end{split}
\end{equation}

By direct computations, we check that $\n'$ satisfies conditions
\eqref{D-can} and therefore it is a natural connection on $\M$.
Its torsion is
\begin{equation}\label{T-can-F}
\begin{split}
T'(x,y,z)&=T^0(x,y,z)-\frac{1}{8}\left\{N(\f^2 z,\f^2 y,\f^2
x)+2N(\f z,\f y,\xi)\eta(x)\right\}_{[x\leftrightarrow y]}.
\end{split}
\end{equation}

We verify immediately that $T'$ satisfies \eqref{T-can} and thus
$\n'$, determined by \eqref{D=nQ} and \eqref{Q-can-F}, is a
$\f$-canonical connection on $\M$.

The explicit expression \eqref{D=nQ}, supported by \eqref{Q0} and
\eqref{NF}, of the $\f$-canonical connection by the tensor $F$
implies that the $\f$-canonical connection is unique.

Immediately we get the following
\begin{prop}\label{prop-Q=Q0}
A necessary and sufficient condition the $\f$-canonical connection
to
coincide with the $\f$B-connection 
is $N(\f\cdot,\f\cdot)=0$.
\end{prop}

Thus, \propref{prop-Q=Q0} and \lemref{lem-Nfifi} imply
\begin{cor}\label{cor-Q=Q0-class}
The $\f$-canonical connection and the $\f$B-connection coincide on
an almost con\-tact B-metric manifold $\M$ if and only if $\M$ is
in the class $\U_0$.
\end{cor}

In \cite{ManIv36}, it is given a classification of the linear
connections on the almost contact B-metric manifolds with respect
to their torsion tensors $T$ in 11 classes $\T_{ij}$. The
characteristic conditions of these basic classes are the
following:
\begin{equation*}\label{Tij}
\begin{split}
\T_{11/12}:\quad
    &T(\xi,y,z)=T(x,y,\xi)=0,\quad 
    T(x,y,z)=-T(\f x,\f y,z)=\mp T(x,\f y,\f z);\\
\T_{13}:\quad &T(\xi,y,z)=T(x,y,\xi)=0,\quad 
            T(x,y,z)-T(\f x,\f y,z)=\sx T(x,y,z)=0;\\
\T_{14}:\quad &T(\xi,y,z)=T(x,y,\xi)=0,\quad 
            T(x,y,z)-T(\f x,\f y,z)=\sx T(\f x,y,z)=0;\\
\T_{21/22}:\quad &T(x,y,z)=\eta(z)T(\f^2 x,\f^2 y,\xi),\quad
                T(x,y,\xi)=\mp T(\f x,\f y,\xi);\\
\T_{31/32}:\quad &T(x,y,z)=\eta(x)T(\xi,\f^2 y,\f^2 z)-\eta(y)T(\xi,\f^2 x,\f^2 z),\quad \\
                &T(\xi,y,z)=\pm T(\xi,z,y)=-T(\xi,\f y,\f z); \\
\T_{33/34}:\quad &T(x,y,z)=\eta(x)T(\xi,\f^2 y,\f^2 z)-\eta(y)T(\xi,\f^2 x,\f^2 z),\\
                &T(\xi,y,z)=\pm T(\xi,z,y)=T(\xi,\f y,\f z); \\
\T_{41}:\quad
&T(x,y,z)=\eta(z)\left\{\eta(y)\hat{t}(x)-\eta(x)\hat{t}(y)\right\}.
\end{split}
\end{equation*}

\section{General Contactly Conformal Group $G$}

In this section we consider the group of transformations of the
$\f$-canonical connection  generated by the general contactly
conformal transformations of the almost contact B-metric
structure.

Let $\M$ be an almost contact B-metric manifold. The general
contactly conformal transformations of the almost contact B-metric
structure are defined by
\begin{equation}\label{Transf}
\begin{array}{c}
    \bar{\xi}=e^{-w}\xi,\quad \bar{\eta}=e^{w}\eta,\quad 
    \bar{g}(x,y)=\al g(x,y)+\bt g(x,\f y)+(\gm-\al)\eta(x)\eta(y),
\end{array}
\end{equation}
where $\al=e^{2u}\cos{2v}$, $\bt=e^{2u}\sin{2v}$, $\gm=e^{2w}$ for
differentiable functions $u$, $v$, $w$ on $M$ \cite{Man4}. These
transformations form a group denoted by $G$.

If $w=0$, we obtain the contactly conformal transformations of the
B-metric, introduced in \cite{ManGri1}. By $v=w=0$, the
transformations \eqref{Transf} are reduced to the usual conformal
transformations of $g$.

Let us remark that $G$ can be considered as a complex conformal
gauge group, i.e. the composition of an almost contact group,
preserving $H$ and a complex conformal transformation of the
complex Riemannian metric $\overline{g^{\C}}=e^{2(u+iv)}g^{\C}$ on
$H$.

Let $\M$ and $(M,\f,\bar{\xi},\bar{\eta},\bar{g})$ be contactly
conformally equivalent with respect to a transformation from $G$.
The Levi-Civita connection of $\bar{g}$ is denoted by
 $\bar{\n}$.
Using the formula
\[
\begin{split}
    2g(\n_xy,z)=xg(y,z)+yg(x,z)-zg(x,y)
               +g([x,y],z)+g([z,x],y)+g([z,y],x),
\end{split}
\]
by straightforward computations we get the following relation
between  $\n$ and $\bar{\n}$
:
\begin{equation}\label{bar-n-n}
\begin{split}
    &2\left(\al^2+\bt^2\right)g\left(\bar{\n}_xy-\n_xy,z\right)=\\
    &=\frac{1}{2}\Bigl\{-\al\bt \left[2F(x,y,\f^2z)-F(\f^2z,x,y)\right]-\bt^2 \left[2F(x,y,\f z)-F(\f z,x,y)\right]\\
    &\phantom{=\frac{1}{2}\Bigl\{}+\frac{\bt}{\gm}\left(\al^2+\bt^2\right)
    \left[2F(x,y,\xi)-F(\xi,x,y)\right]\eta(z)\\
    &\phantom{=\frac{1}{2}\Bigl\{}+2\left(\frac{\al}{\gm}-1\right)\left(\al^2+\bt^2\right)
    F(\f^2x,\f y,\xi)\eta(z)\\
    &\phantom{=\frac{1}{2}\Bigl\{}+2\al(\gm-\al)
    \left[F(x,\f z,\xi)+F(\f^2z,\f x,\xi)\right]\eta(y)\\
    &\phantom{=\frac{1}{2}\Bigl\{}-2\bt(\gm-\al)
    \left[F(x,\f^2z,\xi)-F(\f z,\f x,\xi)\right]\eta(y)\\
    &\phantom{=\frac{1}{2}\Bigl\{}-2\left[\al\D\al(x)+\bt\D\bt(x)\right]g(\f y,\f z)+2\left[\al\D\bt(x)-\bt\D\al(x)\right]g(y,\f
    z)\\
    &\phantom{=\frac{1}{2}\Bigl\{}-\left[\al\D\al(\f^2z)+\bt\D\al(\f z)\right]g(\f x,\f y)+\left[\al\D\bt(\f^2z)+\bt\D\bt(\f z)\right]g(x,\f
    y)\\
    &\phantom{=\frac{1}{2}\Bigl\{}+\left[\al\D\gm(\f^2z)+\bt\D\gm(\f z)\right]\eta(x)\eta(y)\\
    &\phantom{=\frac{1}{2}\Bigl\{}+\frac{1}{\gm}(\al^2+\bt^2)\bigl\{\D\al(\xi)g(\f x,\f y)-\D\bt(\xi)g(x,\f y)\bigr\}\eta(z)\\
    &\phantom{=\frac{1}{2}\Bigl\{}+\frac{1}{\gm}(\al^2+\bt^2)\bigl\{2\D\gm(x)\eta(y)-\D\gm(\xi)\eta(x)\eta(y)\bigr\}\eta(z)\Bigr\}_{(x\leftrightarrow
    y)}.
\end{split}
\end{equation}

Using \eqref{F=nfi} and \eqref{bar-n-n}, we obtain the following
formula for the  transformation by $G$ of the tensor $F$:
\begin{equation}\label{barF-F}
\begin{split}
    &2\bar{F}(x,y,z)=2\al F(x,y,z)+\Bigl\{\bt \left\{F(\f y,z,x)-F(y,\f z,x)+F(x,\f y,\xi)\eta(z)\right\}\\
    &\phantom{+\Bigl\{}+(\gm-\al)\bigl\{\left[F(x,y,\xi)+F(\f y,\f x,\xi)\right]\eta(z)
    +\left[F(y,z,\xi)+F(\f z,\f y,\xi)\right]\eta(x)\bigr\}\\
    &\phantom{+\Bigl\{}-\left[\D\al(\f y)+\D\bt(y)\right]g(\f x,\f z)-\left[\D\al(y)-\D\bt(\f y)\right]g(x,\f
    z)\\
    &\phantom{+\Bigl\{}+\eta(x)\eta(y)\D\gm(\f z)\Bigr\}_{(y\leftrightarrow
    z)}.
\end{split}
\end{equation}

\begin{thm}
The tensor $N(\f\cdot,\f\cdot)$ is an invariant of the group $G$
on any almost contact B-metric manifold.
\end{thm}
\begin{proof}
From \eqref{NF}, \eqref{barF-F}, \eqref{F-prop} and
\eqref{Transf}, it follows the formula for the transformation by
$G$ of the Nijenhuis tensor:
\begin{equation}\label{barN-N}
\begin{split}
    \bar{N}(\f x,\f y,z)=\al N(\f x,\f y,z)&+\bt N(\f x,\f y,\f z)+(\gm-\al) N(\f x,\f
    y,\xi)\eta(z).
\end{split}
\end{equation}
Thus, bearing in mind \eqref{Transf}, we obtain immediately
$\bar{N}(\f x,\f y)=N(\f x,\f y)$.
\end{proof}

According to \lemref{lem-Nfifi}, we establish immediately the
following
\begin{cor}
The class $\U_0$ is closed by the action of the group $G$.
\end{cor}

\begin{prop}\label{prop:bar-n'-n'}
Let  the almost contact B-metric manifolds $\M$ and
$(M,\f,\bar{\xi},\bar{\eta},\bar{g})$ be contactly conformally
equivalent with respect to a transformation from $G$. Then the
corresponding $\f$-canonical connections $\bar{\n}'$ and $\n'$ as
well as their torsions $\bar{T}'$ and  $T'$ are related as
follows:
\begin{equation}\label{bar-n'-n'}
\begin{split}
    \bar{\n}'_xy&=\n'_xy
    -\D u(x)\f^2 y+\D v(x)\f y+\D w(x)\eta(y)\xi\\
    &+\frac{1}{2}\bigl\{\left[\D u(\f^2y)-\D v(\f y)\right]\f^2 x-\left[\D u(\f y)+\D v(\f^2 y)\right]\f
    x\\
    &\phantom{+\frac{1}{2}\bigl\{}-g(\f x,\f y)\left[\f^2p-\f q\right]+g(x,\f y)\left[\f p+\f^2
    q\right]\bigr\},
\end{split}
\end{equation}
\begin{equation}\label{T'TP}
\begin{split}
    \bar{T}'(x,y)&=T'(x,y)
    +\frac{1}{2}\bigl\{2\D w(x)\eta(y)\xi\\
    &\phantom{=T'(x,y)+\frac{1}{2}\bigl\{}
        +\left[\D u(\f^2x)+\D v(\f x)-2\D u(\xi)\eta(x)\right]\f^2 y\\
    &\phantom{=T'(x,y)+\frac{1}{2}\bigl\{}
    +\left[\D u(\f x)-\D v(\f^2 x)+2\D v(\xi)\eta(x)\right]\f y\bigr\}_{[x\leftrightarrow
    y]},
\end{split}
\end{equation}
where $p=\grad{u}$, $q=\grad{v}$.
\end{prop}

\begin{proof}
Taking into account \eqref{fB}, we have the following equality on
$\M$:
\begin{equation}\label{5}
\begin{split}
&g\left(\n^{0}_xy-{\n}_xy,z\right)
=\frac{1}{2}\left\{{F}(x,\f y,z)+{F}(x,\f y,\xi)\eta(z)-2{F}(x,\f
z,\xi)\eta(y)\right\}.
\end{split}
\end{equation}
Then we can rewrite the corresponding equality on the manifold
$(M,\f,\bar{\xi},\bar{\eta},\bar{g})$, which is the image of $\M$
by a transformation from $G$:
\begin{equation}\label{6}
\begin{split}
&\bar{g}\left(\bar{\n}^{0}_xy-\bar{\n}_xy,z\right)
=\frac{1}{2}\left\{\bar{F}(x,\f y,z)+\bar{F}(x,\f
y,\bar{\xi})\bar{\eta}(z)-2\bar{F}(x,\f
z,\bar{\xi})\bar{\eta}(y)\right\}.
\end{split}
\end{equation}

By virtue of \eqref{5}, \eqref{6}, \eqref{barF-F},
\eqref{bar-n-n}, we get the following formula of the
transformation by $G$ of the $\f$B-connection:
\begin{equation}\label{bar-n0-n0}
\begin{split}
    &g\left(\bar{\n}^0_xy-\n^0_xy,z\right)=
    \frac{1}{8}\sin{4v} N(\f z,\f  y,\f x)-\frac{1}{4}\sin^2{2v} N(\f^2z,\f^2 y,\f^2 x)\\
    &-\frac{1}{4}e^{2(w-u)}\sin{2v}N(\f^2 z,\f y,\xi)\eta(x)
    -\frac{1}{4}\left(1-e^{2(w-u)}\cos{2v}\right)N(\f z,\f y,\xi)\eta(x)\\
    &-\D u(x)g(\f y,\f z)+\D v(x)g(y,\f z)+\D w(x)\eta(y)\eta(z)\\
    &+\frac{1}{2}\left[\D u(\f^2y)-\D v(\f y)\right]g(\f x,\f z)-\frac{1}{2}\left[\D u(\f y)+\D v(\f^2 y)\right]g( x,\f z)\\
    &-\frac{1}{2}\left[\D u(\f^2z)-\D v(\f z)\right]g(\f x,\f y)+\frac{1}{2}\left[\D u(\f z)+\D v(\f^2 z)\right]g( x,\f
    y).
\end{split}
\end{equation}

Taking into account \eqref{Q-can-F}, \eqref{barN-N},
\eqref{Transf} and \eqref{bar-n0-n0}, we get \eqref{bar-n'-n'}.
As a consequence of \eqref{bar-n'-n'}, the torsions $T'$ and
$\bar{T}'$ of  $\n'$ and $\bar{\n}'$, respectively, are related as
in \eqref{T'TP}.
\end{proof}

The torsion forms associated with $T'$ of the $\f$-canonical
connection are defined, in a similar way of \eqref{titi}, as
follows:
\begin{equation}\label{t}
\begin{array}{c}
t'(x)=g^{ij}T'(x,e_i,e_j),\quad t'^*(x)=g^{ij}T'(x,e_i,\f e_j),\quad 
\hat{t}'(x)=T'(x,\xi,\xi).
\end{array}
\end{equation}
Obviously, $\hat{t}(\xi)=0$ is always valid.

Using \eqref{t}, \eqref{T-can-F}, \eqref{T0}, \eqref{F-prop} and
\lemref{lem-N}, we obtain that the torsion forms of the
$\f$-canonical connection are expressed by the associated forms
with $F$:
\begin{equation}\label{TD-sledi}
\begin{array}{c}
t'=\frac{1}{2}\bigl\{\ta^*+\ta^*(\xi)\eta\bigr\},\quad
t'^*=-\frac{1}{2}\bigl\{\ta+\ta(\xi)\eta\bigr\},\quad
\hat{t}'=-\om\circ\f.
\end{array}
\end{equation}

The equality \eqref{tata*} and \eqref{TD-sledi} imply the
following relation:
\begin{equation}\label{t't'*}
t'^*\circ\f=-t'\circ\f^2.
\end{equation}

\section{General Contactly Conformal Subgroup $G_0$}

Let us consider the subgroup $G_0$ of $G$ defined by the
conditions
\begin{equation}\label{G0}
    \D u\circ\f^2+\D v\circ\f=\D u\circ\f -\D v\circ\f^2 y=\D u(\xi)=\D
    v(\xi)=\D w\circ\f=0.
\end{equation}
By direct computations, from  \eqref{Fi}, \eqref{Transf},
\eqref{barF-F} and \eqref{G0}, we prove that each of the basic
classes $\F_i$ $(i=1,2,\dots,11)$ of the almost contact B-metric
manifolds
 is closed by the action of the group $G_0$. Moreover, $G_0$ is
the largest subgroup of $G$ preserving the 1-forms $\ta$, $\ta^*$,
$\om$ and the special class $\F_0$.

\begin{thm}
The torsion of the $\f$-canonical connection is invariant with
respect to the general contactly conformal transformations if and
only if these transformations belong to the group $G_0$.
\end{thm}
\begin{proof}
\propref{prop:bar-n'-n'} and \eqref{G0} imply immediately
\begin{equation}\label{bar-n'-n'-G0}
\begin{split}
    \bar{\n}'_xy=\n'_xy
    &-\D u(x)\f^2 y+\D v(x)\f y+\D w(\xi)\eta(x)\eta(y)\xi\\
    &-\D u(y)\f^2 x+\D v(y)\f
    x+g(\f x,\f y)p-g(x,\f y)q.
\end{split}
\end{equation}
The statement follows from \eqref{bar-n'-n'-G0}, or alternatively
from \eqref{T'TP} and \eqref{G0}.
\end{proof}

Bearing in mind the invariance of $\F_i$ $(i=1,2,\dots,11)$ and
$T'$ with respect to the transformations of $G_0$, we establish
that each of the eleven basic classes of the manifolds $\M$ is
characterized by the torsion of the $\f$-canonical connection.
Then we give this characterization in the following
\begin{prop}\label{prop:FiT}
The basic classes of the almost contact B-metric manifolds are
characterized by conditions for the torsion of the $\f$-canonical
connection as follows:
\[
\begin{array}{rl}
\F_1:\; &T'(x,y)=\frac{1}{2n}\left\{t'(\f^2 x)\f^2 y-t'(\f^2 y)\f^2 x
            +t'(\f x)\f y-t'(\f y)\f x\right\}; \\
\F_2:\; &T'(\xi,y)=0,\; \eta\left(T'(x,y)\right)=0,\; T'(x,y)=T'(\f x,\f y),\; t'=0;\\
\F_3:\; &T'(\xi,y)=0,\; \eta\left(T'(x,y)\right)=0,\; T'(x,y)=\f T'(x,\f y);\\
\F_4:\; &T'(x,y)=\frac{1}{2n}t'^*(\xi)\left\{\eta(y)\f x-\eta(x)\f y\right\};\\
\F_5:\; &T'(x,y)=\frac{1}{2n}t'(\xi)\left\{\eta(y)\f^2 x-\eta(x)\f^2 y\right\};\\
\F_6:\; &T'(x,y)=\eta(x)T'(\xi,y)-\eta(y)T'(\xi,x),\;\\
&T'(\xi,y,z)=T'(\xi,z,y)=-T'(\xi,\f y,\f z);\\
\F_{7/8}:\; &T'(x,y)=\eta(x)T'(\xi,y)-\eta(y)T'(\xi,x)+\eta(T'(x,y))\xi,\\
            &T'(\xi,y,z)=-T'(\xi,z,y)=\mp T'(\xi,\f y,\f z)
            =\frac{1}{2}T'(y,z,\xi)=\mp \frac{1}{2}T'(\f y,\f z,\xi);\\
\F_{9/10}:\; &T'(x,y)=\eta(x)T'(\xi,y)-\eta(y)T'(\xi,x),\; \\
            &T'(\xi,y,z)=\pm T'(\xi,z,y)=T'(\xi,\f y,\f z);\\
\F_{11}:\; &T'(x,y)=\left\{\hat{t'}(x)\eta(y)-\hat{t'}(y)\eta(x)\right\}\xi.\\
\end{array}
\]
\end{prop}
\begin{proof}
According to \propref{prop-Q=Q0}, \corref{cor-Q=Q0-class},
\eqref{T0} and \eqref{T-can-F}, we have the following form of the
torsion of the $\f$-canonical connection when $\M$ belongs to the
classes $\F_i$ $(i\in\{1,2,\dots,11\}; i\neq 3,7)$:
\begin{equation*}\label{TD}
\begin{split}
T'(x,y)=T^0(x,y)=&\frac{1}{2}\bigl\{\left(\n_x\f\right)\f
y+\left(\n_x\eta\right)y\cdot\xi+2\eta(x)\n_y\xi\bigr\}_{[x\leftrightarrow
y]}.
\end{split}
\end{equation*}
For the classes $\F_3$ and $\F_7$, we use \eqref{T-can-F} and
equalities \eqref{N-1-11}.

Then, using \eqref{F-prop}, \eqref{TD-sledi}, \eqref{t't'*} and
\eqref{Fi}, we obtain the characteristics in the statement.
\end{proof}

According to the classification of the torsion tensors in
\cite{ManIv36} and \propref{prop:FiT}, we get the following
\begin{prop}\label{prop:FiT'jk}
Let $T'$ be the torsion of the $\f$-canonical connection on an
almost contact B-metric manifold $M=\M$. The correspondence
between the classes $\F_i$ of $M$ and the classes $\T_{jk}$ of
$T'$ is given as follows:
\[
\begin{array}{ll}
M\in\F_1\; \Leftrightarrow \; T'\in\T_{13},\; t'\neq 0; \qquad &
M\in\F_7\; \Leftrightarrow \; T'\in\T_{21}\oplus\T_{32}; \\
M\in\F_2\; \Leftrightarrow \; T'\in\T_{13},\; t'= 0; \qquad &
M\in\F_8\; \Leftrightarrow \; T'\in\T_{22}\oplus\T_{34}; \\
M\in\F_3\; \Leftrightarrow \; T'\in\T_{12}; \qquad &
M\in\F_9\; \Leftrightarrow \; T'\in\T_{33}; \\
M\in\F_4\; \Leftrightarrow \; T'\in\T_{31},\; t'= 0,\; t'^*\neq 0;
\qquad &
M\in\F_{10}\; \Leftrightarrow \; T'\in\T_{34}; \\
M\in\F_5\; \Leftrightarrow \; T'\in\T_{31},\; t'\neq 0,\; t'^*= 0;
\qquad & M\in\F_{11}\; \Leftrightarrow \; T'\in\T_{41}.\\
M\in\F_6\; \Leftrightarrow \; T'\in\T_{31},\; t'= 0,\; t'^*= 0;
\qquad &
\end{array}
\]
\end{prop}

\section{The relation between the almost contact B-metric manifolds
and the K\"ahlerian manifolds with Norden metric}

Firstly, let us consider the  manifold $M^{\times}=M\times\R$,
where $\M$ is a $(2n+1)$-dimensional almost contact B-metric
manifold. The almost complex structure on $M^{\times}$ is defined
(as in \cite{Blair}) by $J(x,a\ddt)=(\f x-a\xi,\eta(x)\ddt)$,
where $x\in\X(M)$, $t$ is the coordinate on $\R$, $\ddt=\dt$ and
$a$ is a differentiable function on $M^{\times}$. Let us consider
the product metric
$g^{\times}\left(x^{\times},y^{\times}\right)=g(x,y)-ab$ for
$x^{\times}=(x,a\ddt)$, $y^{\times}=(y,b\ddt)$ on $M^{\times}$.
Since $g$ is a B-metric, then $g^{\times}$ is a Norden metric,
i.e.
$g^{\times}\left(Jx^{\times},Jy^{\times}\right)=-g^{\times}\left(x^{\times},y^{\times}\right)$.
Thus $(M^{\times},J,g^{\times})$ is an almost Norden manifold. We
consider $M$ as a hypersurface of $M^{\times}$. Then the
Gauss-Weingarten equations are
$\n^{\times}_{x}y=\n_xy-g(Ax,y)\ddt$, $\n^{\times}_{x}\ddt=-Ax$,
where $\n^{\times}$ is the Levi-Civita connection for $g^{\times}$
and $A$ is its Weingarten map. Then we have
\[
\begin{array}{l}
\left(\n^{\times}_xJ\right)y=\left(\n_x\f\right)y-\eta(y)Ax-g(Ax,y)\xi+\left\{\left(\n_x\eta\right)y-g(Ax,\f
y)\right\}\ddt,\;\;\; \left(\n^{\times}_{\ddt}J\right)\ddt=A\xi,
\\
\left(\n^{\times}_xJ\right)\ddt=-\n_x\xi+\f Ax+2\eta(Ax)\ddt,\quad
\left(\n^{\times}_{\ddt}J\right)y=\f Ay-A\f y+\eta(Ay)\ddt.
\end{array}
\]

If we set $(M^{\times},J,g^{\times})$ to be a K\"ahlerian manifold
with  Norden metric, i.e. $\n^{\times}J=0$, we obtain
\[
\left(\n_x\f\right)y=\eta(y)Ax+g(Ax,y)\xi, \qquad
\left(\n_x\eta\right)y=g(Ax,\f y),
\]
\[
\n_x\xi=\f Ax,\qquad \f\circ A=A\circ\f ,\qquad \eta(Ax)=0,\qquad
A\xi=0.
\]
Therefore we have
$$F(x,y,z)=F(x,y,\xi)\eta(z)+F(x,z,\xi)\eta(y),\quad F(x,y,\xi)=
F(y,x,\xi)=-F(\f x,\f y,\xi).$$
Thus, bearing in mind \eqref{Fi}, the manifold $\M$ belongs to the
class  $\F_4\oplus\F_5\oplus\F_6$, which may be viewed as the
class of the almost contact B-metric manifolds of  Sasakian type
as an analogy to the contact metric geometry \cite{Blair}.

Secondly, in \cite{GaMiGr} it is given an example of the
considered manifolds as follows.
Let the vector space
$\R^{2n+2}=\left\{\left(u^1,\dots,u^{n+1};v^1,\dots,v^{n+1}\right)\
|\ u^i,v^i\in\R\right\}$ be considered as a complex Riemannian
manifold with the canonical complex structure $J$ and the metric
$g$ defined by
$g(x,x)=-\delta_{ij}\lm^i\lm^j+\delta_{ij}\mu^i\mu^j$ for
$x=\lm^i\frac{\partial}{\partial
u^i}+\mu^i\frac{\partial}{\partial v^i}$. Identifying the point
$p\in\R^{2n+2}$ with its position vector it is considered the
time-like sphere $S: g(n,n)=-1$ of $g$ in $\R^{2n+2}$, where $n$
is the unit normal to the tangent space $T_pS$ at $p\in S$. It is
set $g(n,Jn)=\tan\psi$,
$\psi\in\left(-\frac{\pi}{2},\frac{\pi}{2}\right)$. Then the
almost contact structure is introduced by
$\xi=\sin\psi.n+\cos\psi.Jn$, $\eta=g(\cdot,\xi)$,
$\f=J-\eta\otimes J\xi$. It is shown that $(S,\f,\xi,\eta,g)$ is
an almost contact B-metric manifold in the class $\F_4\oplus\F_5$.

Since the $\f$-canonical connection coincides with the
$\f$B-connection on any manifold in $\F_4\oplus\F_5$, according to
\corref{cor-Q=Q0-class}, then by virtue of \eqref{T0} we get the
torsion tensor and the torsion forms of the $\f$-canonical
connection as follows:
\[
\begin{array}{l}
T'(x,y,z)=\bigl\{\eta(x)\left\{\cos\psi g(y,\f z)
-\sin\psi g(\f y,\f z)\right\}\bigr\}_{[x\leftrightarrow y]},\\
t'=2n\sin\psi\eta,\qquad t'^*=-2n\cos\psi\eta,\qquad \hat{t}'=0.
\end{array}
\]
These equalities are in accordance with \propref{prop:FiT}.
Moreover, it follows that the statement $T'\in\T_{31}$ is valid,
which confirms \propref{prop:FiT'jk}.



\end{document}